\documentclass[11pt]{amsart}
\advance\hoffset by -0.5 truecm
\usepackage{amssymb}
\usepackage{amsmath}
\usepackage{graphicx}
\usepackage{color}

\newtheorem{Theorem}{Theorem}[section]
\newtheorem{Lemma}[Theorem]{Lemma}
\newtheorem{Corollary}[Theorem]{Corollary}

\newtheorem{Proposition}[Theorem]{Proposition}

\newtheorem{Definition}[Theorem]{Definition}
\newtheorem{Example}[Theorem]{Example}
\newtheorem{Remark}[Theorem]{Remark}

\def \dim{{\mbox {dim}}\,}

\def\R\re
\def \vr{\varphi}
\def \R{{\mathbb R}}
\def \N{{\mathbb N}}

\begin{document}
\title[Solutions]{Isoparametric functions and solutions of Yamabe type equations on manifolds with boundary. }

\author[G. Henry]{Guillermo Henry}\thanks{G. Henry is partially supported by  grant PICT-2016-1302 from ANPCyT and by grant 20020170100330BA from Universidad de Buenos Aires.}
\address{Departamento de Matem\'atica, FCEyN, Universidad de Buenos
	Aires, Ciudad Universitaria, Pab. I., C1428EHA,
	Buenos Aires, Argentina and CONICET, Argentina.}
\email{ghenry@dm.uba.ar}

\author[J. Zuccotti]{Juan Zuccotti}\thanks{J. Zuccotti is supported by a doctoral fellowship of CONICET}
\address{Departamento de Matem\'atica, FCEyN, Universidad de Buenos
	Aires, Ciudad Universitaria, Pab. I., C1428EHA,
	Buenos Aires, Argentina and CONICET, Argentina.}
\email{jzuccotti@dm.uba.ar}


\date{}

\maketitle

\begin{abstract}
Let $(M,g)$ be a compact Riemannian manifold with non-empty boundary. Provided $f$ an isoparametric function of $(M,g)$ 
we prove existence results for positive solutions of the Yamabe equation that are constant along the level sets of $f$.
If $(M,g)$ has positive constant scalar curvature, minimal boundary and admits an isoparametric function we also prove multiplicity results for positive solutions of the Yamabe equation on $(M \times N,g+th) $ where $(N,h)$ is any closed Riemannian  manifold with positive constant scalar curvature.
\end{abstract}

\section{Introduction}

One of the classical generalizations of the celebrated  Yamabe Problem, the so called {\it minimal boundary Yamabe Problem}, was proposed by Escobar \cite{Escobar2} in 1992. Given a compact Riemannian manifold $(M^n,g)$ of dimension $n\geq 3$ with non-empty boundary $\partial M$, the minimal boundary Yamabe problem consists in finding a conformally equivalent metric to $g$ with constant scalar curvature and zero mean curvature on the boundary.  This is a equivalent to obtain, for some constant $c$,  a positive smooth solution of the following boundary value problem  
\begin{equation}\label{EYM}\begin{cases}  a_n\Delta_gu+ s_g u=c u^{p_n-1} & \mbox{on } M, \\ \frac{2}{(n-2)}\frac{\partial u}{\partial \eta}+h_{g}u=0, & \mbox{on } \partial M,\end{cases}
\end{equation}
where $a_n=\frac{4(n-1)}{(n-2)}$, $p_n=\frac{2n}{n-2}$,  $\eta$ is the outward unit normal vector field along $\partial M$, $s_g$ is the scalar curvature and $h_g$ is the mean curvature  of $\partial M$. Indeed, if $u>0$ is a smooth solution of Equation \eqref{EYM}  then the Riemannian metric $g_u=u^{p_n-2}g$  satisfies that \[s_{g_u}=c\ \mbox{and}\ h_{g_u}=0.\]
Escobar proved in \cite{Escobar2} that unless  $\dim(M) \geq 6$, $(M,g)$ is non-locally conformally flat,  $\partial M$  is umbilic and  the  Weyl tensor vanishes identically on it, there exists a positive smooth solution to the Equation \eqref{EYM}. The remaining case was settled by Brendle and Chen \cite{Brendle_Chen} provide the positive mass theorem holds. For a non-variational approach see the recent paper by Xu \cite{Xu}.

Positive solutions to Equation \eqref{EYM}  are in one to one correspondence with positive critical points of the functional  
\begin{equation}\label{functYB}
J_g(u):=\frac{\int_Ma_n|\nabla u|^2_g+s_gu^2dv_g+2(n-1)\int_{\partial M}h_gu^2d\sigma}{\Big(\int_Mu^{p_n}dv_g\Big)^{\frac{2}{p_n}}}
\end{equation}
where $dv_g$ and $d\sigma$ are the the volume elements induced by $g$ on $M$ and $\partial M$,  respectively.

In this setting we define the {\it Yamabe constant of} $(M,g)$ as 
\begin{equation*}Y(M,\partial M,[g]):=\inf_{u\in C^{\infty}(M)-\{0\}}J_g(u).
\end{equation*}
Let us denote by $[g]$ the conformal class of $g$, that is the set of Riemannian metrics of the form $\phi g$ with $\phi>0$ smooth. It is well known that  $Y(M,\partial M,h)=Y(M,\partial M,g)$ for any $h\in [g]$, hence the Yamabe constant is  a conformal invariant and we are going to denote it by $Y(M,\partial M,[g])$.

Let $(S^n_+,g^n_0)$ be the $n-$dimensional upper half sphere endowed with the standard metric. Cherrier proved in \cite{Cherrier} that if  
\begin{equation}\label{IneqYamabeCte}Y(M,\partial M,[g])< Y(S^n_+,\partial S^n_+,[g^n_{0}]),
\end{equation}
holds, then $Y(M,\partial M,[g])$ is attained by a positive smooth function.  Escobar in \cite{Escobar2} proved that Inequality \eqref{IneqYamabeCte}  holds in most of the cases. Therefore, there is a positive smooth solution to the Equation \eqref{EYM}  with minimal Yamabe energy.

Given $c_1,\ c_2\in \R$,  a  more general problem is to ask whether is possible to find $\tilde{g}\in [g]$ such that   $s_{\tilde{g}}=c_1$ and $h_{\tilde{g}}=c_2$.
Let $L_g:=a_n\Delta +s_g$ and   $B_g:=\frac{2}{(n-2)}\frac{\partial }{\partial \eta}+h_g$ be the  conformal Laplacian  and the boundary operator respectively. The scalar curvature of $g_u$ and the mean curvature of $\partial M$ with respect to $g_u$  are given by
\begin{equation}\label{curvaturaescalar}s_{g_u}=u^{1-p_n}L_g(u) 
\end{equation}
and
\begin{equation}\label{curvaturamedia}
h_{g_u}=u^{-\frac{p_n}{2}}B_g(u).
\end{equation}
Therefore, the problem of conformally deforming a metric $g$ to one with constant scalar curvature $c_1$ and constant mean curvature equals to $c_2$ on $\partial M$ is equivalent to  find a positive solution of 
\begin{equation}\label{EYG}\begin{cases}  L_g(u)=c_1 u^{p_n-1} & \mbox{on } M, \\ B_g(u)=c_2u^{\frac{p_n}{2}} & \mbox{on } \partial M.\end{cases}
\end{equation}

The minimal boundary Yamabe problem is the special case when $c_1\in \R$ and $c_2=0$.  Another important case if when $c_1=0$ and $c_2\in \R$.  This problem  was addressed by Escobar in \cite{Escobar1} in order to generalized Riemann mapping theorem to higher dimension and it is known as the {\it constant boundary mean curvature Yamabe problem}.  
Let us consider the functional \[Q_{g}(u):=\frac{\int_Ma_n|\nabla u|^2_g+s_gu^2dv_g+2(n-1)\int_{\partial M}h_gu^2d\sigma}{\Big(\int_{\partial M}u^{p_n^{\partial}}d\sigma\Big)^{\frac{2}{p_n^{\partial}}}}\]
where $p_n^{\partial}=\frac{2(n-1)}{(n-2)}$,  and the conformal invariant
\[\tilde{Y}(M,\partial M, [g]):=\inf_{u\in C^{\infty}(M)-\{0\}}Q_g(U).\]

Escobar proved (see \cite{Escobar1} and \cite{Escobar4}), under the assumption  that $\tilde{Y}(M,\partial M, g)$ is finite, that  if either $n\geq 6$ and $\partial M$ has a nonumbilic point or $n\geq 6$, $M$ is locally conformally flat and $\partial M$ is umbilic or $n=4,5$ and $\partial M$ is umbilic or $n=3$, then there is a scalar flat metric in $[g]$ with constant mean curvature on $\partial M$. Actually, the metric he found is of the form $g_u=u^{p_n-2}g$ where $u$ is a minimizer of  $\tilde{Y}(M,\partial M, g)$. The cases that were not considered by Escobar were covered in several articles by Marques \cite{Marques} and \cite{Marques2}, Almaraz \cite{Almaraz} and \cite{Almaraz2}, Chen \cite{Chen}, and Mayer and Ndiaye \cite{Mayer-Ndiaye} using a non-variational approach.

 If $Y(M,\partial M,[g])=0$, then there exists a positive solution of Equation \eqref{EYG} with $c_1=c_2=0$ (see \cite{Escobar1} and \cite{Escobar4}). Han  and  Li  conjectured in \cite{Han-YYli} that if $Y(M,\partial M,[g])>0$, Equation \eqref{EYG} admits a positive smooth solution for any $c_1>0$ and $c_2\in \R$ (or equivalently $c_1=1$ and $c_2\in \R$).  This conjecture was proved for a wide range of manifolds by the combined works of Han and Li \cite{Han-YYli}, \cite{Han-YYli2},  Chen and Sun \cite{Chen-Sun} and Chen, Ruan and Sun  \cite{Chen-Ruan- Sun}. On the other hand, when $Y(M,\partial M,[g])<0$, Chen, Ho and Sun \cite{Chen-Ho-Sun} showed that there exists  a unique positive smooth solution of  Equation \eqref{EYG} for any $c_1<0$ and $c_2<0$.

\vspace{0.5 cm}

Let $(M,g)$ be a closed Riemannian manifold. A non-constant smooth function $f:M\longrightarrow \R$ is called {\it isoparametric} if there exist $a$ and $b$  smooth functions such that  
\begin{equation}\label{Cgradient}\| \nabla f \|^2 = b \circ f
\end{equation}
and
\begin{equation}\label{Claplacian}
\Delta_{g} f = a\circ f.
\end{equation}

The geometric meaning of these conditions is the following. Equation \eqref{Cgradient} implies that the regular level sets, which are called  {\it isoparametric hypersurfaces}, are equidistant to each other, while both Equation \eqref{Cgradient} and Equation \eqref{Claplacian} imply that regular level set are constant mean curvature hypersurfaces. Wang proved in \cite{Wang} that the only critical level sets of $f$ are the maximum $t_+$ and the minimum $t_-$, and their preimages $M_+=f^{-1}(t_+)$ and $M_-=f^{-1}(t_-)$ are submanifolds of $M$ that we called {\it focal submanifolds}. Moreover, he proved that  each isoparametric hypersurface is a tube over either of the focal submanifolds.
This implies that there are topological obstructions to the existence of isoparametric functions. Indeed, a closed Riemannian manifold that admits an isoparametric function must be diffeomorphic to the union of two disc bundles over the focal submanifolds (see Miyaoka \cite{Miyaoka}). On the other hand, by a result of Qian and Tang \cite{Qian-Tang}  we know that for any closed manifold $M$ that admits a Morse-Bott function $f$ with critical level set $M_+\cup M_-$ where $M_-$ and $M_+$ are both closed connected submanifolds of codimensions at least 2, there exists a Riemannian metric $g$ such that $f$ is an isoparametric function of $(M,g)$.

The theory of isoparametric hypersurfaces is very rich. It started, motivated by the modelling of  wavefronts,  in the first decades of the twenty century with the classification of isoparametric hypersurfaces of Euclidean and the Hyperbolic space by Segre \cite{Segre} and Cartan respectively. The classification of isoparametric hypersurfaces on the sphere turned out to be a very hard problem (see \cite{Yau}). It was initiated by Cartan \cite{Cartan2} in 1939 and was completed recently by Chi \cite{Chi1} but a lot of researchers contributed significatively to the solution of this problem. For a survey in the history on the classification of isoparametric hypersurfaces we refer the reader to \cite{Chi2} and the references therein (see also \cite{Cecil}).   

 The isoparametric theory was applied in several contexts. Only to mention some of them,  see Tang and Yan \cite{Tang-Yan2} on the Yau's conjecture on the first eigenvalue of the Laplace-Beltrami operator of a minimal hypersurface in the unit sphere; or  Tang and Yan \cite{Tang-Yan1} on Willmore submanifolds (see also Qian et al. \cite{Qianetal}). 

In the last years they have appeared several articles that make use of  isoparametrics functions to produce both positive and changing sign solutions of the Yamabe equation on closed Riemannian manifolds.  See for instance, Henry and Petean \cite{Henry-Petean}, Henry \cite{Henry}, Fern\'andez and Petean \cite{Fernandez-Petean},  de la Parra et al. \cite{delaParra}, among others. In this article we are going to exploit some of the ideas of mentioned papers and applied  them to the setting of compact manifolds with boundary. 

Let us define isoparametric function in the setting  of compact manifold with non-empty boundary.

\begin{Definition}We say that a non-constant smooth function $f:M\longrightarrow \R$ is an {\it isoparametric function} of $(M,g)$ if it satisfies \eqref{Cgradient},  \eqref{Claplacian} and is locally constant on $\partial M$.   
\end{Definition}

Let $f$ be an isoparametric function of $(M,g)$.  We are going to address the following question: 

\vspace{0.3 cm}
 Does there exist a Riemannian metric $\tilde{g}$  conformal to $g$ such that $f$ is an isoparametric function of $(M,\tilde{g})$ as well and satisfies either $\tilde{g}$ has constant scalar curvature and minimal boundary or $\tilde{g}$ is scalar flat and $\partial M$ is a constant mean curvature hypersurface?

\vspace{0.3 cm}

We will give a positive answer to this question whenever all the connected components of  $M_-$  and $M_+$  have positive dimension.

Let $$k(f):=min \big\{\dim(M_{-}), \dim(M_{+}) \big\}.$$
 By $S_f$ we denote the space of functions that are constant along the level sets of $f$  and by $[g]_f$ the set of  metrics $h\in [g]$  such that $f$ is an isoparametric function  of $(M,h)$ as well.

 Our main result is the following:
 
 \begin{Theorem}\label{A} Let $(M^n,g)$ be a compact manifold  $(n\geq 3)$ with boundary and let $s\geq 1$. Let $f$ be an isoparametric function of $(M,g)$ and  assume that $s_{g}\in S_f$. Let us denote with $(a)$, $(b)$ and $(c)$ the following assumptions:
 	
 	\begin{itemize}
 		\item[(a)] $k(f)\geq n-2$.  
 		\item[(b)] $k(f)<(n-2)$ and $s<\frac{2(n-k(f))}{n-k(f)-2}$.
 		\item[(c)] $k(f)<(n-2)$ and $s<\frac{2(n-k(f)-1)}{n-k(f)-2}$.
 	\end{itemize}
We have,

 	\begin{enumerate}

	\item[i)]  If either $(a)$ or $(b)$ holds, then there exists $($for some $c\in \R$$)$ a positive smooth function $u \in S_f$ that is  a solution of    
	\begin{equation}\label{ThmA}\begin{cases}  a_n\Delta_{g}u+s_{g}u=c u^{s-1} & \mbox{on } M, \\ \frac{2}{n-2}\frac{\partial u}{\partial \eta}+h_{g}u=0 & \mbox{on } \partial M. \end{cases}
\end{equation}

\item[ii)] Let assume that $\tilde{Y˘}(M,\partial M, [g])$ is finite.  If either $(a)$ or $(c)$ is fulfilled, then there exists $($for some $c\in \R$$)$ a positive smooth function $u \in S_f$ such that
\begin{equation}\label{ThmB}\begin{cases}  a_n\Delta_{g}u+s_{g}u=0 & \mbox{on } M, \\ \frac{2}{n-2}\frac{\partial u}{\partial \eta}+h_{g}u=cu^{s-1} & \mbox{on } \partial M. \end{cases}
\end{equation}

\end{enumerate}

 \end{Theorem}

 	\begin{Remark} The conformal invariant $\tilde{Y}(M,\partial M, [g])$ might be not finite. Let $\lambda_1^D$, $\lambda_1^B(g)$, and $\lambda_1^L(g)$ be the first eigenvalues of the followings eigenvalue problems: 
 		\begin{equation}\label{Dirichleteigenvalue}\begin{cases}  L_g(u)=\lambda_1^D(g)u & \mbox{on } M, \\ u=0 & \mbox{on } \partial M,\end{cases}
 		\end{equation}
 		\begin{equation}\label{Beigenvalue}\begin{cases}  L_g(u)=0 & \mbox{on } M, \\ B_g(u)=\lambda_1^B(g)u & \mbox{on } \partial M\end{cases}
 		\end{equation}
\begin{equation}\label{Leigenvalue}\begin{cases}  L_g(u)=\lambda_1^L(g)u & \mbox{on } M, \\ B_g(u)=0 & \mbox{on } \partial M\end{cases}
 		 	 	\end{equation}
   
   It can be seen that if the first  eigenvalue of Dirichlet eigenvalue problem, 
$\lambda_1^D$, is negative then $Q(M,\partial M, [g])=-\infty$ (see \cite{Escobar3}). 
We have that \[ sign\big(\tilde{Y}(M,\partial M, [g])\big)=sign\big(\lambda_1^B(g)\big)=sign\big(\lambda_1^L(g)\big)=sign\big(Y(M,\partial M, [g])\big).\]  Therefore, if $Y(M,\partial M,[g])\geq 0$  then $\tilde{Y}(M,\partial M, [g])$ is finite. For instance, the latter situation holds if $[g]$ admits a metric of non-negative scalar curvature.
 	\end{Remark}

 A consequence of the above  theorem above is:

 \begin{Corollary}\label{Corolario1}
 	Let assume that $k(f)\geq 1$ and $\tilde{Y}(M,\partial M, [g])$ is finite. Then there exist  Riemannian metrics $h_1$ and $h_2$ that belong to $[g]_f$ such that
 	\begin{itemize}
 		\item[i)] $h_1$ has constant scalar curvature and  $\partial M$ is minimal.  
 		\item[ii)] $h_2$ is scalar flat and the mean curvature is constant on $\partial M$.
 		 	\end{itemize}
 	
 \end{Corollary}


If $(M,g)$ admits an isoparametric function with $K(f)\geq 1$  and  $Y(M, \partial M,[g])>0$ then Han-Li conjecture holds.

\begin{Theorem}\label{Han-Li conjecture} Let assume that $Y(M, \partial M,[g])>0$ and $k(f)\geq 1$, then for any $c\in \R$ there exists a metric $h_{c}\in [g]_f$ with constant scalar curvature 1 and constant mean curvature $c$ on $\partial M$.  
\end{Theorem}
\vspace{0.5 cm}

By Corollary \ref{uniqueness}, see Section \ref{Section3}, if $\partial M$ is connected  the metric $h_2\in [g]_f$ provided by Corollary \ref{Corolario1} is unique. However, for the minimal boundary Yamabe problem we might have multiplicity among the metrics in $[g]_f$.

   Let $(M^m,g)$ be a compact Riemannian manifold ($m\geq 2$) with non-empty   boundary, positive constant scalar curvature and minimal boundary.  Let $(N^n,h)$ be any closed Riemannian manifold with positive constant scalar curvature and $t>0$, then $(M\times N,g(t)=g+th)$  is a manifold with constant scalar curvature $s_{g(t)}=s_g+\frac{1}{t}s_h$ and boundary $\partial M\times N$. The mean curvature  on $\partial M\times N$  is zero.  Assume that $f$ is an isoparametric function of $(M,g)$.  After we have done the identification between real valued functions on $M$ and real valued functions on $M\times N$  that do not depend on $N$ is easy to verify that an isoparametric function of $(M,g)$ is an isoparametric function of $(M\times N,g(t))$ as well.

With the assumptions mentioned above we obtain:

\begin{Theorem}\label{localbifurcation}
	There exists a sequence of positive real numbers $\{t_i\}_{i\in \N  }$  that   tends to zero
	 and satisfies that for any sequence of positive numbers  $\{\varepsilon_i\}_{i\in \N}$    there exist a sequence $\{u_i\}_{i\in \N}$, with $u_i\in S_f$,  and  a sequence of positive numbers $\{\gamma_i\}_{i\in \N}$ such that 
	\begin{itemize}
		\item  $|\gamma_i-t_i|<\varepsilon_i$,
		\item $u_i^{p_{m+n}-2}g(\gamma_i)\in [g(\gamma_i)]_f$ is a  constant scalar curvature with minimal boundary.
	\end{itemize}
\end{Theorem}

\vspace{0.5 cm}

The article is organized as follows. In Section \ref{Isoparametric function} we discuss some facts about the theory of isoparametrics functions and isoparametric hypersurfaces on manifolds with boundary.  In Section \ref{Section3} we prove Theorem \ref{ThmA}, Corollary \ref{Corolario1} and Theorem \ref{Han-Li conjecture}. Finally, in Section \ref{Section4} we prove Theorem \ref{localbifurcation}.

\section{Isoparametric functions}\label{Isoparametric function}

In this section we revisit some well known facts on the theory of isoparametric functions on closed Riemannian manifolds and we point out the similarities and some differences that arise in the setting of compact manifold with non-empty boundary.

Let $f:M\longrightarrow [t_-,t_+]$ be an isoparametric function. We denote with $M_t$ the level set $f^{-1}(t)$ and with $B_f:=\big\{p\in M:\ \nabla f|_p=0\big\}$.

Let us see a few examples of isoparametric functions on manifolds with boundary.

\begin{Example}\label{exmpl1}
	Let $-1\leq c_1< c_2\leq 1$. We write  $S_{c_1,c_2}^n$ the points of the unitary sphere $S^n\subset \R^{n+1}$   such that $c_1\leq x_{n+1}\leq c_2$. We consider $S_{c_1,c_2}^n$ endowed with the metric induced by $g^n_0$.   If $|c_i|<1$, then $\partial S_{c_1,c_2}^n= \Sigma_1\cup \Sigma_2$ where $\Sigma_i=\{x:\ x_{n+1}=c_i\}$. 
	The functions $f(x)=x_{n+1}$  and  $h(x)=x_1^2+\dots+ x^2_n-x^2_{n+1}$ fulfill  conditions \eqref{Cgradient} and \eqref{Claplacian} and they are constant along $\Sigma_i$. Therefore $f$ and $h$ are isoparametric functions of $(S_{c_1,c_2}^n,g^n_0)$.
	
	The function  $f_j(x)=x_{j}$ with $j<n+1$ satisfies Conditions \eqref{Cgradient} and \eqref{Claplacian},  however $\Sigma_i$   is not  included in any level set if $|c_i|<1$.
	\end{Example}

\begin{Example}\label{example2} Let $\tilde{M}$ be a closed manifold and let $\Sigma\subset \tilde{M}$ be a connected  hypersurface that is the boundary of some region $M$.  Let assume that $\Sigma$ is the regular level set  of a Morse-Bott function $f$, such that the critical level sets of $f$ are exactly  $M_+$  and $M_{-}$. In addition we assume that  $M_+$  and $M_{-}$ are connected and have codimension at least 2. Then, as we have mentioned in the Introduction,  by (Theorem 1.1, \cite{Qian-Tang})  there exists a Riemannian metric $g$ on  $\tilde{M}$ such that $f$ is an isoparametric function of $(\tilde{M},g)$. Hence, $f|_M$ is an isoparametric function of $(M, g|_{M})$. 
\end{Example}

If $f$ is an isoparametric function of a closed Riemannian manifold $(M,g)$, Wang proved in (Lemma 3, \cite{Wang}) that $f(B_f)=\{t_-,t_+\}$. However, when $\partial M\neq \emptyset$  $f(B_f)$ might be a proper subset of $\{t_-,t_+\}$, even $B_f=\emptyset$. For instance, let $f:S_{c_1,c_2}^n\longrightarrow [-1,1]$ as in Example \ref{exmpl1}. If $-1<c_1<c_2<1$, $B_f=\emptyset$; if $c_1=-1$ and $c_2<1$ or if $c_1<-1$ and $c_2=1$  $f(B_f)=-1$ and $f(B_f)=1$, respectively; $h(B_h)=\{0, 1\}$ for $h(x)=x_1^2+\dots+ x^2_n-x^2_{n+1}$ on $S^n_{0,1}$.

Let $f$ be an isoparametric function of $(M,g)$ and let $[t_1,t_2]$ an interval such that $f(B_f)\cap [t_1,t_2]=\emptyset$.  With the same proof of (Lemma 1, \cite{Wang}) we can see that for any $x\in M_{t_1}$ and $y\in M_{t_2}$ it yields
 \begin{equation}\label{parallel}
d(x,M_{t_2})=d(y,M_{t_1})=\int_{t_1}^{t_2}\frac{df}{\sqrt{b(f)}}.
\end{equation}
Equation \eqref{parallel} says that the regular level sets are equidistant to each other. The shortest path between $M_{t_1}$ and $M_{t_2}$ are the integral curves of the tangent field $(b\circ f)^{-\frac{1}{2}} \nabla f$.

If $M_{t_2}$ is the only critical level set between the level sets $\{M_t\}_{t\in [t_1,t_2]}$  then 
\begin{equation}\label{parallel2}
d(M_{t_1},M_{t_2})=\lim_{t\to t_2^-}\int_{t_1}^t\frac{df}{\sqrt{b(f)}}.
\end{equation}
This equation implies that $f(B_f)\subseteq \{t_-,t_+\}$. Indeed, if there exists $x\in B_f$ such that $f(x)=t_0\in(t_-,t_+)$, let $\varepsilon>0$ such that $f(B_f)\cap [t_0-\varepsilon, t_0]=t_0$.  On one hand from \eqref{parallel2} we have that distance between $M_{t_0-\varepsilon}$ and $M_{t_0}$ is equals to $\lim_{t\to t_0^-}\int_{t_0-\varepsilon}^t\frac{df}{\sqrt{b(f)}}$. On the other hand, $b$ attains a minimum at $t_0$. Since $t_0$ is an interior point of $[t_-,t_+]$,  then $b'(t_0)=0$. Therefore, in  neighborhood of  $t_0$, the Taylor expansion of $b$ is  $c(t-t_0)^2+o((t-t_0)^2)$ for some constant $c$. However,  this implies that \[\lim_{t\to t_0^-}\int_{t_0-\varepsilon}^t\frac{df}{\sqrt{b(f)}}=+\infty\] which is a contradiction.

For  $x\in M$  let  $H_x:T_xM\longrightarrow T_xM$ be the linear operator induced  by the Hessian of $f$, $H_f$, at $x$. We write with $E_{\lambda}(x)$ the eigenspace associated to the eigenvalue $\lambda$ of $H_x$.   
 
 The proof of the following proposition is similar to the one of Theorem A point a) in \cite{Wang}. We refer the reader to \cite{Wang} for the proof.

\begin{Proposition} $M_+$ and $M_-$ are (possible non-connected) submanifolds of $M$. If $M_+\cap B_f=\emptyset$ ($M_-\cap B_f=\emptyset$) then $M_+$ ($M_-$) is a hypersurface. If $M_+\subseteq B_f$ ($M_-\subseteq B_f$) and $x\in M_{+}$ ($x\in M_-$) the dimension of the connected component of $M_+$ $(M_-)$ where $x$ belongs  is equals to $\dim(E_{0}(x))$.   	
\end{Proposition}

We call $M_+$ and $M_-$ focal submanifolds although they might not be focal sets (see Proposition 2.1, \cite{Ge-Tang}).

\begin{Remark} If a connected component $C$ of a focal submanifold does not belong to $\partial M$, then $C\subseteq B_f$. 
	
\end{Remark}

Let $P$ be a submanifold of $M$ and $x\in P$. We denote with $N^1_xP$ the fiber of the unitary normal bundle of $P$ at $x$. Given $t>0$  we say that $S_tP$  is a {\it tube of radius $t$} over $P$ if $$S_tP:=\big\{y=exp_x(tv):\ \mbox{where}\ x\in P\ \mbox{and}\ v\in N^1_xP \big\}.$$ 
By $S_{\leq t}P$ we denote  
$$S_{\leq t}P=\big\{y=exp_x(sv):\ \mbox{where}\ x\in P,\ 0\leq s\leq t\ \mbox{and}\ v\in N^1_xP \big\}.$$ 

Let $M_{t_0}$ be a regular level set and let $r$ be the distance between $M_{t_0}$ and the focal submanifold $V$. Let $M_{t_0}^{r}\subset M_{t_0}$ be the union of the connected components  of  $M_{t_0}$   whose distant to $M_+$  is $r$.  We define $\phi_{r}:=M_{t_0}^{r}\longrightarrow V$ by 
\[\phi_{r}(x)=\exp_x(r\xi_N(x)).\]
where $\xi_N(x)= \frac{\nabla f}{\sqrt{b(f(x))}}$ if $V=M_+$ or $-\xi_N(x)$ if $V=M_-$.  Following \cite{Wang}  (see Lemma 7 and Lemma 9, \cite{Wang}) it yields
  \[\phi_{r_0}^{-1}(V )=M_{t_0}^{r}\subseteq M_{t_0}\]
which implies that  $M_{t_0}^{r}$ is a tube over $M_+$. We obtained the following:

\begin{Proposition}	A regular level set is a union of tubes over either of the focal submanifolds.  
\end{Proposition}

For closed manifolds any regular level set of an isoparametric function is a tube over any of the focal submanifolds (Theorem A, \cite{Wang}). However, this might be not the case for manifolds with non-empty boundary. For instance, the regular level sets of $f(x)=x_{n+1}^2$ on  $S_{-{\frac{1}{2}},1}^n$ are not a tube over $M_+$.

\begin{Remark} 
	In this setting there are some manifolds that admit isoparametric functions such that some regular level sets are not tubes over neither of the focal submanifolds. For instance, let us consider the cylinder $M= \mathbb{S}^1 \times [-\frac{\pi}{2},\frac{7}{4}\pi]$ and $f(x,y,z)=\cos z$. Since 
	$df= - \sin z \, dz$, where $d \theta$ is the angular component,  we have that 
	\begin{equation*}
		||\nabla f||^2=||df||^2= \sin^2z=1-\cos^2z=1-f^2(x,y,z),
	\end{equation*}
 and 
 \begin{equation*}\Delta_g f=- \cos z=-f(x,y,z).
 \end{equation*}
 $f$ is locally constant on $\partial M$. Therefore, $f$ is an isoparametric function, however, for any $t \in \left(0, \frac{ \sqrt{2}}{2}\right)$, $M_t$ is not a tube over the focal submanifolds.
\end{Remark}

\subsection{Curvature of the level sets}

Let $M_t$ with $t \in (t_-,t_+)$. Let $S_t$ denote the shape operator on $M_t$ and $h_t(x)$ the mean curvature of $M_t$ at $x$ with respect to the unit normal $\xi_N(x)$.   Then, for all $X,Y \in T_xM_t$ we have that
\begin{equation}\label{ShapeHess}
	\langle S_t(X), Y \rangle= - \frac{H_f(X,Y)}{|\nabla f|}.
\end{equation} 
 To see this, recall that the unit normal vector field to $M_t$ is given by $\xi_N(x)= \frac{\nabla f}{\sqrt{b \circ f}}$. 
 We denote with $\rho= \sqrt{b \circ f}=|\nabla f|$,  then we have $\nabla f=\rho \xi_N$. Since the shape operator $S_t$  satisfies $\langle S_t(X),Y \rangle=-\langle \nabla_X \xi_N, Y \rangle$ for any $X,Y \in T_xM_t$ we have  that
\[H_f(X,Y)=
	\langle \nabla_X \nabla f , Y \rangle= \langle \nabla_X(\rho \xi_N), Y \rangle = \langle X(\rho)\xi_N+ \rho \nabla_X \xi_N, Y \rangle  
\]
\[
=X(\rho)\langle \xi_N, Y\rangle + \rho \langle \nabla_X \xi_N, Y\rangle=-\rho \langle S_t(X), Y \rangle
\]
where for the last equality we used that $\langle \xi_N, Y\rangle=0$ since $Y\in T_xM_t$ and $\xi_N$ is orthogonal to $M_t$.

\begin{Proposition}
	Let $t \in (t_-,t_+)$. Then,
	\begin{equation}\label{MeanCLG}
		h_t= \frac{1}{\rho^2}\left(  \rho \Delta_gf - \langle \nabla f, \nabla \rho \rangle \right).
		\end{equation} 
	\end{Proposition}
\begin{proof}
	Let $x \in M_t$ and let $\{E_1, \hdots, E_{n-1}\}$ be an orthonormal basis of $T_xM_t$.  Since $f$ is constant along $M_t$, then we have $E_i(f)=0$ and $E_i(\rho^{-1})=0$ for all $1 \leq i \leq n-1$.  Using (\ref{ShapeHess}) 
 we have that $H_f(X,Y)=- \rho \langle S_t(X),Y \rangle= \rho \langle \nabla_X \xi_N,Y \rangle$ for $X, Y \in T_xM_t$.   We can compute the Laplace-Beltrami operator of $f$ in the following way
	\[
	\Delta_gf=-\operatorname{tr}_gH_f=- \sum_{i=1}^{n-1}H_f(E_i,E_i)-H_f(\xi_N,\xi_N)
	\]
	\[
	=-\sum_{i=1}^{n-1} \rho \langle \nabla_{E_i} \xi_N,E_i \rangle -\langle \nabla_{\xi_N} \nabla f, \xi_N \rangle=-\sum_{i=1}^{n-1} \rho \langle \nabla_{E_i} \left( \rho^{-1} \nabla f  \right) ,E_i \rangle  -\langle \nabla_{\xi_N} \nabla f, \xi_N \rangle
	\]
	\[
	= -\sum_{i=1}^{n-1}\langle \nabla_{E_i} \nabla f,E_i \rangle -\langle \nabla_{\xi_N} \nabla f, \xi_N \rangle.
	\]
By definition, 
\[
h_t=-\sum_{i=1}^{n-1}\langle \nabla_{E_i} \xi_N, E_i \rangle=-\sum_{i=1}^{n-1} \langle \nabla_{E_i} (\rho^{-1}\nabla f),E_i \rangle = - \frac{1}{\rho} \sum_{i=1}^{n-1}\langle \nabla_{E_i}\nabla f, E_i \rangle 
\]
\[
=\frac{1}{\rho} \left( \Delta_gf-\langle \nabla_{\xi_N} \nabla f, \xi_N \rangle   \right)= \frac{1}{\rho} \left( \Delta_gf-\frac{1}{\rho^2} \langle \nabla_{\nabla f} \nabla f, \nabla f  \rangle \right)
\]
\[
=\frac{1}{\rho} \left( \Delta_g f- \frac{1}{2\rho^2}(\nabla f)(||\nabla f||^2)  \right)= \frac{1}{\rho} \left( \Delta_g f- \frac{1}{2\rho^2}(\nabla f)(\rho^2)  \right)
\]
\[
= \frac{1}{\rho}\left( \Delta_g f- \frac{2\rho}{2\rho^2}(\nabla f)(\rho)  \right)= \frac{1}{\rho} \left( \Delta_g f- \frac{1}{\rho}(\nabla f)(\rho) \right)
\]
\[
=\frac{1}{\rho}\left( \Delta_gf-\frac{1}{\rho}\langle \nabla f, \nabla \rho \rangle \right)=\frac{1}{\rho^2}\left(  \rho \Delta_gf - \langle \nabla f, \nabla \rho \rangle \right).
\]
\end{proof}
This allows us to prove the following
\begin{Proposition}
	$M_t$ has constant mean curvature for all $t \in (t_-,t_+)$.
\end{Proposition}
\begin{proof}
	Using the last proposition and Equations (\ref{Cgradient}) and (\ref{Claplacian}) we have
	\[
	h_t= \frac{1}{\rho^2}\left(  \rho \Delta_gf - \langle \nabla f, \nabla \rho \rangle \right)= \frac{1}{\rho^2} \left(  \rho \, \cdot a  \circ f - \langle \nabla f, \nabla \sqrt{||\nabla f||^2} \rangle \right)
	\]
	\[
	=\frac{1}{\rho^{2}}\left( \rho \cdot a \circ f -\langle \nabla f, \nabla \sqrt{ b \circ f} \rangle \right)= \frac{1}{\rho^2}\left(\rho \cdot a \circ f - \left\langle \nabla f, \frac{1}{2\rho} (b' \circ f )\nabla f  \right\rangle \right)
	\]
	\[
	=\frac{1}{\rho}a \circ f- \frac{1}{2\rho} b' \circ f- = \left( \frac{ 2a-b'}{2 \sqrt{b} } \right) \circ f.
	\]
	Since $f$ is constant along $M_t$, the last expression does not depend on $x \in M_t$.
\end{proof}
\begin{Remark} If $(M,g)$ is a manifold with non-constant mean curvature on the boundary, then  $(M,g)$ does not admit an isoparametric function.
\end{Remark}

In \cite{Ge-Tang2}, Ge and Tang proved that the focal submanifolds of an isoparametric function $f$ on a complete Riemannian manifold $M$ are minimal submanifolds. The strategy they used was the following. Using the results of Wang (\cite{Wang}) we know that the focal submanifolds are smooth manifolds and every $M_t$ is a tubular hyper-surface over either of the focal submanifolds. To see, for instance, that $M_+$ is a minimal submanifold, take a point $x \in M_+$ and $v \in N_x^1M_{+}$ a unitary normal vector to $T_xM_+$. If we consider $\eta_v(t):= \exp_x(tv)$ the unique geodesic through $x$ in the direction of $v$, we can take Fermi coordinates $(x_1, \hdots, x_n)$ centered at $x$ with respect to $M_+$ such that $\frac{\partial}{\partial x_n}|_{\eta_v(t)}=N_{|_{\eta_v(t)}}$, where $N$ is the (outward) unitary vector field of $M_t$. Using this coordinate system we can relate $T_v$, the shape operator of $M_+$ in the direction of $v$, with $S_t$ by obtaining the Taylor expansion of $S_t$ in terms of $t$. Let $T^v_{ab}$, $1\leq a,b \leq m=\dim M_+$, and $S_{ij}$, $1 \leq i,j \leq n$ be the coefficients in this coordinate system of $T_v$ and $S_t$, respectively.  By several straightforward calculations, Ge and Tang proved that at any point of $\eta_v(t)=\exp_x(tv) \in M_t$ we have
\[
S_t=(S_{ij})=\begin{pmatrix}
T_v+tA+O(t^2) & &tB+O(t^2)& &O(t^2) \\
tC+O(t^2)     & &-t^{-1}I+D+O(t^2)& &O(t^2)\\
0& &0& &0
\end{pmatrix}
\]
where the matrices $A,B,C, D$ are given by \[A:= \left( \left\langle R_{v \frac{\partial}{\partial x_a}}v, \frac{\partial }{\partial x_b}\right\rangle+ \sum_c T^v_{ac}T^v{cb}  \right),\]
\[
D:=\frac{1}{3} \left( \left\langle R_{v,\frac{\partial}{\partial x_l}}v, \frac{\partial}{\partial x_k} \right\rangle \right)
\]
\[
B:= \left(  \left\langle R_{v\frac{\partial}{\partial x_a}}v,  \frac{\partial }{\partial x_k} \right\rangle \right)
\]
\[
C:= \frac{1}{3}\left(  \left\langle R_{v \frac{\partial}{\partial x_l} }v  , \frac{\partial}{\partial x_b}       \right\rangle \right)
\]
for $a,b,c \in \left\{ 1, \hdots, m \right\}$, $l,k \in \{ m+1, \hdots, n-1\}$, and $O(t^2)$ are matrices of order less than $2$. Therefore, the principal curvatures of $M_t$ at $\eta_v(t)=\exp_x(tv)$ are the eigenvalues of 
\[
\overline{S}_t:= 
\begin{pmatrix}
T_v+tA+O(t^2)& &tB+O(t^2) \\
tC+O(t^2)& &t^{-1}I+D+O(t^2)
\end{pmatrix}.
\]
Since the mean curvature of $M_t$ is constant then
\[
h_t=\operatorname{tr}(\overline{S}_t)= \operatorname{tr}(T_v)- \frac{n-m-1}{t}+t( \operatorname{tr}A+\operatorname{tr}D)+O(t^2),
\]
does not depend on $x \in M_+$ and $v \in N_x^1M_+$.  Therefore, for all $v \in N_p^1M_+$ we have $\operatorname{tr}(T_v)=\operatorname{tr}(T_{-v})=-\operatorname{tr}(T_v)$. Hence, $\operatorname{tr}(T_v)=0$ for all $v \in N_x^1M_+$ and all $x \in M_+$, which implies that $M_+$ is a minimal submanifold. 

In the setting of manifolds with boundary this is not true since some of the focal varieties could contain some component of the boundary which is a regular hypersurface of constant mean curvature but not necessarily minimal. Nonetheless, if the focal variety contains no component of the boundary, we can proceed exactly as in the case of closed manifolds. Hence, we can guarantee the following
\begin{Proposition}
    If all the components of the focal submanifold $V$ has codimension at leas 2, then $V$ is a minimal submanifold.
\end{Proposition}


\section{Proof of Theorem \ref{A}, Corollary \ref{Corolario1}, and Theorem \ref{Han-Li conjecture}}\label{Section3} 

 Let $f$ be an isoparametric function of $(M,g)$. We denote by $H^q_{1,f}(M)$ the completion of $S_f$  with respect  to  the Sobolev norm $$\|u \|_{1,q}=\big(\int_M|\nabla u|^qdv_g+\int_Mu^qdv_g\big)^{\frac{1}{q}}.$$

\subsection{Proof of Theorem \ref{A} and Corollary \ref{Corolario1}} $\,$ 

Theorem \ref{A} follows from  Lemma \ref{fSobolevembedding} which is an adaptation of Lemma 6.1 in \cite{Henry} to the setting of compact manifolds with boundary.

\begin{Lemma}\label{fSobolevembedding} Let $(M^n,g)$ be a compact Riemannian manifold with non-empty boundary  and let $f$ be an isoparametric function with $k(f)\geq 1$. Let $s\geq 1$, we have that
	\begin{enumerate}
		\item[i)] if $q\geq n-k(f)$, then the inclusion map of $H^q_{1,f}(M)$  into $L^s(M)$ and the trace map $tr:H^q_{1,f}(M)\longrightarrow L^s(\partial M)$  are both continuous and compact.
		\item[ii)] If $q< n-k(f)$,  the inclusion map of $H^q_{1,f}(M)$  into $L^s(M)$ is continuous if  \[s \leq \frac{q(n-k(f))}{n-k(f)-q},\] and it is compact if the  strict inequality holds. 
		\item[iii)]  If $q< n-k(f)$,  then $tr:H^q_{1,f}(M)\longrightarrow L^s(\partial M)$  is a continuous map if \[s \leq \frac{q(n-1-k(f))}{n-k(f)-q},\] and it is compact if the above inequality is strict.

\end{enumerate}

\end{Lemma}

The argument to prove the lemma above is similar to the proof of Lemma 6.1 in \cite{Henry}.  For convenience of the reader we sketch the proof here. The key idea  is to make use of the geometry of the foliation induced by the isoparametric function $f$ in order to obtain an improvement of the Sobolev embedding theorems for functions that belong to $S_f$.  As one can check,  if $q<n$  and  $k(f)\geq 1$ we get either $n-k(f)-q\leq 0$ or $q(n-k(f))/(n-k(f)-q)>q^*:=qn/(n-q)$ and $q(n-1-k(f))/(n-k(f)-q)>q^{\partial,*}:=q(n-1)/(n-q)$.  Therefore, for  a sufficiently small positive $\varepsilon$, Lemma \ref{fSobolevembedding} says that $H^q_{1,f}(M)$ is compactly embedded at the same time into $L^{q^*+\varepsilon}(M)$ and $L^{q^{\partial,*}+\varepsilon}(\partial M)$.

\begin{proof}
	Let us denote with $\R^l_+$ the upper half-space of the Euclidean  space $\R^l$ endowed with   the  restriction of the Euclidean metric $g^l_e$.  For $x\in M$, we denote with $k(x)$ the dimension of the connected component of $M_{f(x)}$ where  $x$ is contained.
By the assumptions of the lemma, $k(x)\geq 1$ for any $x\in M$.	 Since the leafs  of the foliation induced by the isoparametric function, $M_t$,   are equidistant to  each other (see Section \ref{Isoparametric function}), given $x\in M$ we can use Fermi coordinates centered at $x$ to construct a coordinate system $(W_x,\varphi_x)$  such that 
\begin{itemize}
\item[(a)] $x\in W_x$.
\item[(b)]$\varphi_x(W_x)=U\times V$, where $U$ and $V$ are  open sets of $\R^{k(x)}$ and $\R^{n-k(x)}_{+}$, respectively.
\item[(c)]  $U\times \pi(\varphi_x(y))\subset \varphi_x(M_{f(y)}\cap W_x)$ for any $y\in W_x$, where $\pi$ is the projection of $\R^{k(x)}\times \R^{n-k(x)}_{+}$ onto $\R^{n-k(x)}_{+}$.
\item[(d)]   There exists a constant $a_x>0$  such that   $a_x^{-1}dv_{g^n_e}\leq dv_g\leq a_xdv_{g^n_e}$ in $W_x$.   
\item[(e)] If $x \in \partial M$ and $\left\{ \frac{\partial}{\partial \varphi_{x}^1}, \hdots, \frac{\partial}{\partial \varphi_{x}^n} \right\}$ is the basis of $T_xM$ induced by the coordinate system, then $\frac{\partial}{\varphi_{x}^n}$ points in the direction of $\eta$. 

\end{itemize}

Since $M$ is a compact manifold, $M$ admits a finite open cover by these charts.  Let  $\big\{(W_{i},\varphi_i)\big\}_{i=1}^N$ be one of these covers where $(W_{i},\varphi_i)$  is centered at $x_i$ and $\varphi_i(W_i)=U_i\times V_i$. Given a smooth function $u$ in $S_f$ we define $u_i:V_i\subset\R^{n-k(x_i)}_+\longrightarrow \R$ by
\[u_i(y)=u(\varphi_i^{-1}(x,y))\]
 where $x$ is any point in $U_i$. 
 These functions are well defined by  property (b). Indeed,  if $b=\pi(\varphi_i(y))$  and $(a,b) \in U_i \times \{b \}$, then $\varphi_i^{-1}(a,b) \in M_{f(y)} \cap W_i$ and $u|_{W_i}$ is constant along $M_{f(y)}\cap W_i$.  
 
 Since $dv_{g|_{\partial M}}= i^{*}( \iota_{\eta} dv_g )$, where $i: \partial M \hookrightarrow M$ is the inclusion map and $\iota_{\eta}$ is the interior multiplication,  we also have  that $(a_x)^{-1} dv_{g^{n-1}_e} \leq dv_{g|_{\partial M}} \leq a_x dv_{g^{n-1}_e}$. 
 Using  property (c) we are able to compare the $L^s-$norm of $u$ and $u_i$ and their gradients.  Indeed, there exist positive constants $A_i$, $B_i$ and $C_i$ (that depend only on $(W_i,\varphi_i)$) such that 
\begin{equation}\label{Lsineq1}
A_i\|u_i\|_{L^s(V_i)}\leq \|u\|_{L^s(W_i)}\leq B_i \|u_i\|_{L^s(V_i)},
\end{equation}
\begin{equation}\label{Lsineq2}
\|\nabla u_i\|_{L^s(V_i)}\leq C_i\|u\|_{L^s(W_i)}
\end{equation}
and
\begin{equation}\label{Lsineqboundary}
\|u|_{\partial M}\|_{L^s(\partial W_i)}\leq D_i\|u|_{\partial V_i}\|_{L^s(\partial V_i)}.
\end{equation}

Taking into account that $k(f)\leq k(x_i) $ for any $i$,  the assumptions i) and ii) imply that either $q\geq n-k(x_i)$ or $q<n-k(x_i)$ and $s\leq \frac{(n-k(x_i))q}{n-k(x_i)-q}$. Since $V_i\subset \R^{n-k(x_i)}_{+}$ is a bounded domain  that satisfies the cone condition  is well known  (see  for instance \cite{Adams}) that  the inclusion of $H^q_1(V_i)$ into $L^s(V_i)$ is a  continuous map.
Therefore, for some positive constant $E_i$ and $F_i$ we have that 
\begin{equation}\label{SobolevIneq}
\|u\|_{L^s(W_i)}\leq B_i \|u_i\|_{L^s(V_i)}\leq E_i \|u_i\|_{H^q_1(V_i)}\leq F_i \|u\|_{H^q_1(W_i)},
\end{equation}
where in the second and in the last inequality we used Inequalities \eqref{Lsineq1} and \eqref{Lsineq2}, respectively. Finally, we obtain
\[\|u\|_{L^s(M)}\leq \sum_{i=1}^{N}\|u\|_{L^s(W_i)}\leq \sum_{i=N}^rF_i\|u\|_{H^q_1(W_i)}\leq F\|u\|_{H^q_1(M)}\]
which says that the inclusion  of $H^q_{1,f}(M)$ into $L^s(M)$ is a  continuous map

Note that with these assumptions, we obtain that  $H^q_{r,f}(M)$ is continously included into $L^s(M)$ for any $r\geq 1$.

By the trace embedding theorem for Euclidean domains (see for instance \cite{Adams}), we get that
\[
||u_{|_{\partial M}}||_{L^s(\partial W_i)} \leq D_i ||{u_i}_{|_{\partial W_i}}||_{L^s(\partial V_i)} \leq H_i \left( ||\nabla u_i||_{L^q(V_i)}+||u_i||_{L^q(V_i)}    \right) \]
\[\leq J_i \left( ||\nabla u||_{L^q(W_i)}+||u||_{L^q(W_i)}  \right)
\]
whether $n-k(x_i) \leq q $ or $n-k(x_i) >  q$ and $1 \leq s \leq \frac{q(n-1-k(x_i))}{n-q-k(x_i)}$. 
 If $n-k(f) \leq q$ or $n-k(f)>q$ and $1 \leq s \leq \frac{q(n-1-k(f))}{n-q-k(f)}$ this holds for all $1 \leq i \leq m$, and we have that
\[
||u_{|_{\partial M}}||_{L^s(\partial M)} \leq \sum_{i=1}^{m}||u_{|_{\partial W_i}}||_{L^s(\partial W_i)} \leq J \sum_{i=1}^{m} \left(  ||\nabla u||_{L^q(W_i)}+ ||u||_{L^q(W_i)} \right) \]
\[\leq \tilde{J}\left(  ||\nabla u||_{L^q(M)}+||u||_{L^q(M)} \right)
\]
which proves the continuity of the trace map under the assumptions of iii).

Now we are going to prove the compactness of the inclusion map  of $H^q_{1,f}(M)$  into $L^s(M)$ when either
\begin{equation}\label{strict1}
	q\geq n-k(f)
	\end{equation}
	or 
	 \begin{equation}\label{strict2}
	 q<n-k(f)\ \mbox{and}\  s<\frac{q(n-k(f))}{n-k(f)-q}.  
	\end{equation}

In order to prove the compactness is enough to show that any bounded sequence $\{u_j\}$ of $H^q_{1,f}(M)$ has a convergent subsequence in $L^s(M)$.
Let us consider  a partition of the unity $\{\phi_i\}_{i=1}^N$ subordinated to the open cover $\{(W_i,\varphi_i)\}$ such that \[\phi_i(\varphi_i^{-1}(\tilde{x},y))= \phi_i(\varphi_i^{-1}(\bar{x},y))\] for any $\tilde{x}$ and  $\bar{x}$ $\in U_i$ and $y\in V_i$ .  For  $u\in S_f$ we define  the compact supported functions $u^i: V_i\longrightarrow \R$  as 
\[u^i(y):=\phi_i(\varphi_i^{-1}(y)).u_i(y). \]

Let $\{u_j\}_{j\in \N}$ be a bounded sequence in  $H^q_{1,f}(M)$, then by  \eqref{Lsineq1} and \eqref{Lsineq2} we see that the sequence
\[\{u^i_j\}_{j\in \N}\]
is a bounded sequence of $H_1^q(V_i)$.

Assume that either \eqref{strict1} or \eqref{strict2} holds. Then for any $1\leq i\leq N$, either $q\geq n-k(x_i)$ or $q<n-k(x_i)$. If the latter holds, it yields  $q<n-k(f)$ and one can check that $s<\frac{q(n-k(x_i))}{n-k(x_i)-q}$ holds as well. Any set $V_i$ fulfills the cone condition, therefore by the Rellich-Kondrachov Theorem (see Theorem 6.3, \cite{Adams}) there is a subsequence $\{u_{j_k}\}$ such that $\{u_{j_k}^i\}$ is a Cauchy sequence in $L^s(V_i)$. By Inequality \eqref{Lsineq1}  $\{\phi_i u_{j_k}\}$  is a Cauchy sequence in $W_i$ as well.  Since $\{\phi_i\}_{i=1}^r$ is a partition of the unity,  $\{u_{j_k}\}_{k\in \N}$ converges in $L^s(M)$.

The proof of the compactness of $tr:H^q_{1,f}(M)\longrightarrow L^s(\partial M)$ under assumption iii) is similar.    If $q \geq n-k(f)$ or $q < n-k(f)$ and $1 \leq s < \frac{q(n-1-k(f))}{n-k(f)-q}$, then the same inequalities hold for all $1 \leq i \leq N$ after replacing $k(f)$ for $k(x_i)$. Since $H_1^q(V_i) \hookrightarrow L^s(\partial V_i)$ is compact, then we have a subsequence $\{u_j\}$ such that $\{{(u_{j,i})}_{|_{\partial V_i}}\}$ is a Cauchy sequence in $L^s(\partial V_i)$ for all $1 \leq i \leq N$. By Property $(d)$ of the coordinate systems $\{(W_i, \varphi_i)\}$ and since $\partial M= \cup_{i=1}^N \partial W_i =\cup_{i=1}^N \partial \varphi_i^{-1}(U_i \times V_i)= \cup_{i=1}^{N}\varphi_i^{-1}(U_i \times \partial V_i) $, we have that $(\rho_i u_j)_{|_{\partial M}}$ is a Cauchy sequence in $L^{s}(\partial M)$ for all $1 \leq i \leq N$, and since $\{\rho_i\}_{i=1}^m$ is a partition of unity, we obtain that $\{u_j|_{\partial M}\}$ converges strongly in $L^s(\partial M)$.

\end{proof}

Let us consider the functionals $J_{g}^s:H^2_{1}(M)-\{0\}\longrightarrow \R$ and  $Q_{g}^s:H^2_{1}(M)-\{0\}\longrightarrow \R$ defined by 
\[J_{g,f}^s(u):=\frac{E(u)}{\Big(\int_Mu^{s}dv_g\Big)^{\frac{2}{s}}}\]
and
\[Q_{g}^s(u):=\frac{E(u)}{\Big(\int_{\partial M}u^{s}d\sigma\Big)^{\frac{2}{s}}}, \]
where \[E(u):=\int_Ma_n|\nabla u|^2_g+s_gu^2dv_g+2(n-1)\int_{\partial M}h_gu^2d\sigma.\]

Note that $J^{p_n}_g=J_g$ and $Q^{p_n^{\partial}}_g=Q_g$.

We are going to consider \[Y^s_f(M,\partial M,g):=\inf_{u\in H^2_{1,f}(M)-\{0\}}J_{g,f}^{s}(u),\]
\[\tilde{Y}^s(M,\partial M,g):=\inf_{u\in H^2_{1}(M)-\{0\}}Q_g^{s}(u)\]
and \[\tilde{Y}^s_f(M,\partial M,g):=\inf_{u\in H^2_{1,f}(M)-\{0\}}Q_g^{s}(u).\]
It follows from definitions  that $\tilde{Y}^{p_n^{\partial}}(M,\partial M,g)=\tilde{Y}(M,\partial M,g)$ and 
\begin{equation}\label{comparacionY_fY}
\tilde{Y}^s(M,\partial M,g)\leq \tilde{Y}^s_f(M,\partial M,g).
\end{equation}
It is well known that positive critical points of $J_{g,f}^s|_{H^2_{1,f}(M)-\{0\}}$ and  $Q_{g,f}^s|_{H^2_{1,f}(M)-\{0\}}$ satisfy Equation \eqref{ThmA} and Equation \eqref{ThmB} respectively (see for instance \cite{Escobar2} and \cite{Escobar1}). Hence to prove Theorem \ref{A} is sufficient to show that $Y^s_f(M,\partial M,[g])$ and $\tilde{Y}^s_f
(M,\partial M,[g])$ are attained by smooth positive functions.

\begin{proof}[Proof of Theorem \ref{A}]

Assume that either $(a)$ or (b) hold.  Let $\{u_i\}_{i\in \N}$ be a minimizing sequence of non-negative functions with  $\|u_i\|_{2,1}=1$.  Taking $q=2$ by Lemma \ref{fSobolevembedding} we have that  $H^2_{1,f}(M)$ is compactly embedded in $L^s(M)$. Therefore, there is subsequence of $\{u_i\}$ that converges to a nonzero function $u\in S_f$ that minimizes $Y^s_f(M,\partial M,g)$. Actually,  using standard elliptic regularity theory it can be seen that $u$ is  smooth. By the strong maximum principle, $u$ is positive in the interior of $M$. By the Hopf boundary point Lemma, $u$ is also positive along the boundary. This proves i).

Now let assume that $\tilde{Y}(M,\partial M,[g])>-\infty$.  We claim that $\tilde{Y}^s(M,\partial M,g)$ is finite for any $2\leq s\leq p_n^{\partial}$. Therefore,  Inequality \eqref{comparacionY_fY} implies that $\tilde{Y}^s_f(M,\partial M,g)$ is finite.

If $\tilde{Y}(M,\partial M,[g])\geq 0$, then $E(u)\geq 0$ for any $u\in H^2_1(M)$. Hence,  $\tilde{Y}^s_f(M,\partial M,g)\geq 0$ for any $2\leq s \leq p_n^{\partial}$. 

If $s'\leq s$, by H\"older inequality we have that 
\[\Big(\int_{\partial M}u^{s'}d\sigma\Big)^{\frac{2}{s'}}\leq\Big(\int_{\partial M}u^{s}d\sigma\Big)^{\frac{2}{s}} vol(\partial M)^{\frac{2(s-s')}{ss'}}.\]  
Therefore, if $E(u)<0$ and $s\geq 2$, $Q_g^{s}(u)\geq Q_g^{2}(u)
vol(\partial M)^{\frac{s-2}{s}}$. Hence 
\begin{equation}\label{ineqBeigenvalue}\tilde{Y}^s(M,\partial M,g)\geq \lambda_1^B(g)vol(\partial M)^{\frac{s-2}{s}}\end{equation}
whenever $\tilde{Y}(M,\partial M,[g])< 0$.
On the other hand, it is well known (see \cite{Escobar3}) that $\tilde{Y}(M,\partial M,[g])>-\infty$ implies that $\lambda_1^B(g)>-\infty$. Thus, \eqref{ineqBeigenvalue} implies that   $\tilde{Y}^s(M,\partial M,g)$ is finite when $\tilde{Y}(M,\partial M,[g])$ is finite and negative and this proves the claim.

Let assume that in addition either $(a)$ or $(c)$ hold as well.  Lemma \ref{fSobolevembedding} says that the embedding of $H^2_{1,f}(M)$ into $L^s(\partial M)$ is compact. Using a similar argument to the one used in the proof of item $i)$ we obtain that $\tilde{Y}^s_f(M,\partial M,[g])$ is attained by a positive smooth function that is constant along the level sets of $f$.

\end{proof}

Let $f$ be an isoparametric function of $(M,g)$ and let us consider the Riemannian metric $h=\phi g$   where $\phi=\psi\circ f$ with $\psi>0$.   
By straightforward computations we obtain that
\begin{equation*}\label{laplacianConfMetric}\Delta_hf=-\frac{(n-2)\psi'(f)}{2\phi^2}\|\nabla_gf\|^2_g+\frac{1}{\phi}\Delta_g(f) =-\frac{(n-2)\psi'(f)b(f)}{2\phi^2}+\frac{a(f)}{\phi},\end{equation*}
and 
\begin{equation*}\label{normgradientConfMetric}\|\nabla_h f\|_h^2=\frac{1}{\phi}\|\nabla_g f\|_g^2=\frac{b(f)}{\phi}.
\end{equation*}
Therefore, $f$ is an isoparametric function of $(M,h)$ as well, that is $h\in [g]_f$. Actually, using a similar argument as in the proof of Proposition 3.1 in \cite{Henry} it can be seen that any metric in $[g]_f$ is of this form. Hence, in order to prove Corollary \ref{Corolario1} is is enough  to show  that there exist $u_1$ and $u_2$ positive smooth functions of $S_f$ that are solutions of Equation \eqref{EYG} with $c_2=0$ and $c_1=0$ respectively.

\begin{proof}[Proof of Corollary \ref{Corolario1}]

	Let us assume that $k(f)<(n-2)$, otherwise we are done taking in Theorem \ref{A}  $s_1=p_n$ and $s_2=p_n^{\partial}$. Since $k(f)\geq 1$, note that  $2n(n-k(f)-2)<2(n-k(f))(n-2)$, then $p_n<2(n-k)/(n-k-2)$ and assumption $(b)$ of Theorem \ref{A}  is fulfilled for $s_1=p_n$. Therefore, there exists a positive smooth function $u_1\in S_f$ that is solution of Equation \eqref{ThmA} with $s_1=p_n$. Also note that  $2(n-1)(n-k(f)-2)<2(n-k(f)-1)(n-2)$. Thus,  assumption $(c)$ of Theorem \ref{A}  is fulfilled for $s_2=p_n^{\partial}$. Hence, there exists a  positive smooth function $u_2\in S_f$ that is solution of Equation \eqref{ThmB} with $s_2=p^{\partial}_n$.
\end{proof}

\begin{Remark} Let $(M,g)$ be compact Riemannian manifold with $\partial M\neq \emptyset$  and constant positive scalar curvature and let $(N,h)$ be any closed Riemannian manifold with constant positive scalar curvature.  If   $f$ is an isoparametric function of $(M,g)$, we mentioned in the Introduction that $f$ is an isoparametric function of $(M\times N,g+h)$ as well, but now all the level sets of $f$ have positive dimension. Therefore, by Corollary \ref{Corolario1} there exist $\mu_1$ and $\mu_2$ in $[g+h]_f$ such that  $\mu_1$ is of  constant scalar curvature  and $\partial M\times N$ is minimal  with respect to $\mu_1$ and $\mu_2$ is scalar flat and $\partial M\times N$ is a constant mean curvature hypersurface with respect to $\mu_2$.
	\end{Remark}

\begin{Proposition}\label{alosumouna} Let $(M,g)$ be a compact Riemannian manifold with a non-empty connected boundary that admits an isoparametric function $f$. There exists at most one unit volume scalar flat metric in $[g]_f$ such that the boundary has constant mean curvature.  
	\end{Proposition}

\begin{proof} Let assume that there are two metrics $g_1$ and $g_2$ that satisfy the statement of the proposition. We can write $g_2=u^{p_n-2}g_1$ where $u$ is a positive smooth function that belongs to $S_f$. Then $u$ satisfies  
	\begin{equation*}\begin{cases}  \Delta_{g_1}u=0 & \mbox{on } M, \\ \frac{2}{n-2}\frac{\partial u}{\partial \eta}+h_{g_1}u=h_{g_2} u^{\frac{p_n}{2}} & \mbox{on } \partial M.\end{cases}
	\end{equation*}
	Since $\partial M$ is connected and $u\in S_f$,  $u$ is constant on $\partial M$. On the other hand $u$ is an harmonic function, then it must be constant on $M$. Let $u\equiv c$,  then $1=vol(g_2)=c^{p_n}$, hence $g_2=g_1$.
\end{proof}

By Corollary \ref{Corolario1}  and  Proposition \ref{alosumouna} we obtain:

\begin{Corollary}\label{uniqueness} Let $(M,g)$ with $\partial M$ connected and $f$ an isoparametric function.  If $k(f)\geq 1$ and $\tilde{Y}(M,\partial M, [g])$ is finite then there exists only one metric in $[g]_f$  with unit volume, zero scalar curvature and constant mean curvature on $\partial M$.  
\end{Corollary}

\subsection{Proof of Theorem \ref{Han-Li conjecture}} $\,$

When $Y(M,\partial M, [g])>0$ we can state Theorem \ref{A} in a more general fashion.

\begin{Proposition}\label{Han-liconjecture-1} Let $(M^n,g)$ be a compact manifold  $(n\geq 3)$ with boundary and $Y(M,\partial M, [g])>0$. Let $p,q\geq 1$ and $q<p$.  Let $f$ be an isoparametric function of $(M,g)$ and  assume that $s_{g}\in S_f$. Let us denote with $(a)$, $(b)$ and $(c)$ the following assumptions:
 	
 	\begin{itemize}
 		\item[(a)] $k(f)\geq n-2$.  
 		\item[(b)] $k(f)<(n-2)$ and $p<\frac{2(n-k(f))}{n-k(f)-2}$.
 		\item[(c)] $k(f)<(n-2)$ and $q<\frac{2(n-k(f)-1)}{n-k(f)-2}$.
 	\end{itemize}
  Either if $(a)$ holds or $(b)$ and $(c)$ hold, then for any $c_1>0$ and $c_2\in \R$ there exists  a positive smooth function $u \in S_f$ that is  a solution of    
	\begin{equation}\label{EqHan-Li-sub}\begin{cases}  a_n\Delta_{g}u+s_{g}u=c_1 u^{p-1} & \mbox{on } M, \\ \frac{2}{n-2}\frac{\partial u}{\partial \eta}+h_{g}u=c_2u^{q-1} & \mbox{on } \partial M. \end{cases}
\end{equation}
 \end{Proposition}

\begin{proof}
For $a>0$, let us define    
\[B^{a,b}_{p,q}:=\Big\{ u \in H^2_1(M) : a \int_M|u|^p dv_g+b  \int_{\partial M} |u|^{q}d\sigma_g =1 \Big\},\] 
and by $B^{a,b}_{p,q,f}$ we denote the set $B^{a,b}_{p,q}\cap S_f$. 
It is well known (see for instance \cite{Araujo}) that  positive critical points of the functional $E$ restricted to  $B^{a,b}_{p,q}$ are positive solution of Equation \eqref{EqHan-Li-sub} with 
\[c_1=\frac{apE(u)}{ap\int_{M}u^{p}dv_g+bq\int_{\partial M}u^qd\sigma_g} \]
and
\[c_2=\frac{bqE(u)}{2(n-1)\Big(ap\int_{M}u^{p}dv_g+bq\int_{\partial M}u^qd\sigma_g\Big)}.\]

Let \[Y^{a,b}_{p,q}(M,\partial M, [g]):= \inf_{B^{a,b}_{p,q}}  E(u)\]  and 
\[Y^{a,b}_{p,q;f}(M,\partial M, [g]):= \inf_{B^{a,b}_{p,q,f}} E(u).\]

Since $a>0$,  we have that $-\infty <Y^{a,b}_{p,q}(M,\partial M, [g])\leq Y^{a,b}_{p,q;f}(M,\partial M, [g])$.

Any positive minimizer of $Y^{a,b}_{p,q;f}(M, \partial M, [g])$ satisfies Equation \eqref{EqHan-Li-sub} for
\[c_1=\frac{apY^{a,b}_{p,q;f}(M,\partial M, [g])}{ap\int_{M}u^{p}dv_g+bq\int_{\partial M}u^qd\sigma_g} 
\]
and 
\[c_2=\frac{bqY^{a,b}_{p,q;f}(M,\partial M, [g])}{2(n-1)\Big(ap\int_{M}u^{p}dv_g+bq\int_{\partial M}u^qd\sigma_g\Big)}.\]

To see the existence of such minimizer, let $\{u_i\}\in B^{a,b}_{p,q,f}$ be a minimizing sequence of non-negative functions. It can be seen that any sequence of bounded energy is uniformly bounded in  $H^2_1(M)$ (see Proposition 2.4 in \cite{Escobar5}, therefore $\{u_i\}$ is a bounded in $H^2_{1,f}(M)$.   Either if $(a)$ holds, or $(b)$ and $(c)$ hold, we have that $H^2_{1,f}(M) \hookrightarrow L^p(M)$ and $H^2_{1,f}(M) \hookrightarrow L^q(\partial M)$ are both compact embeddings. Therefore, there is a subsequence $\{u_{i_k}\}$ that converges to a nonzero function $u \in S_f \cap B^{a,b}_{p,q}$ that minimizes  $Y^{a,b}_{p,q;f}(M, \partial M, [g])$. By standard regularity results, we have that $u$ is smooth. Since $u$ is nonzero, by the strong maximum principle, $u$ is positive in the interior of $M$ and by the Hopf boundary point Lemma, $u$ is also positive along the boundary.
\end{proof}


\begin{proof}[Proof of Theorem \ref{Han-Li conjecture}] For any $a>0$ and $b\in \R$, taking $p=p_n$ and $q=p^{\partial}_n$ it follows from Proposition \ref{Han-liconjecture-1} that exists a positive smooth minimizer $u_{ab}$ of  $Y^{a,b}_{p_n,p^{\partial}_n}(M,\partial M, [g])$. After a normalization we obtain a smooth positive solution $v_{a,b}$ of equation
\begin{equation}\begin{cases}  L_g(v_{ab})=v_{ab}^{p_n-1} & \mbox{on } M, \\ 
B_g(v_{ab})=c_{ab}v_{ab}^{p^{\partial}_n-1} & \mbox{on } \partial M \end{cases}
\end{equation}
where 
\[c_{a,b}=\frac{b}{\sqrt{2n(n-2)a}}\Big[\frac{Y^{a,b}_{p_n,p^{\partial}_n;f}(M,\partial M, [g])}{A(u_{ab})}\Big]^{\frac{1}{2}}.\]
with 
\[A(u_{ab})=ap_n\int_Mu^{p_n}_{ab}dv_g+bp^{\partial}_n\int_{\partial M}u^{p^{\partial}_n}_{ab}d
\sigma_g.\]

Note that 

\begin{equation}\begin{cases}  p_n^{\partial}\leq A(u_{ab}) \leq p_n  & \mbox{if}\ b\geq 0, \\  p_n\leq A(u_{ab}) & \mbox{if }\ b<0. \end{cases}
\end{equation}
Then we have that
\begin{equation}\label{cotainf cte }\frac{b}{2n\sqrt{a}}\Big[ Y^{a,b}_{p_n,p^{\partial}_n;f}(M,\partial M, [g])  \Big]^{\frac{1}{2}}\leq c_{a,b}.
\end{equation}
If $b>0$ we also have that 
\begin{equation}\label{cotasup cte}c_{ab}\leq \frac{b}{2\sqrt{n(n-1)a}}\Big[ Y^{a,b}_{p_n,p^{\partial}_n;f}(M,\partial M, [g])  \Big]^{\frac{1}{2}}.
 \end{equation}

For any $u\in C^{\infty}(M)$ there exists a unique $\lambda_{ab}>0$ such that $\lambda_{ab} u\in B^{a,b}_{p_n,p^{\partial}_n}$. Actually, $\lambda_{a,b}$ is the unique solution of $F_{ab}^u(t)=1$ where 
$F_{ab}^u:\R_{\geq 0}\longrightarrow \R$ is defined by
\[F_{ab}^u(t)=(a\int_Mu^{p_n}dv_g)t^{p_n}+(b\int_Mu^{p_n^{\partial}}d\sigma_g)t^{p_n^{\partial}}.\]
Note that $F_{ab}^u(0)=0$ and $F_{ab}^u$ is an increasing function if $b\geq 0$ and it has a unique  global minimum  if $b<0$.

If $a_2\geq a_1$ we have that
\begin{equation}\label{monoticity of Yab}
Y^{a_1,b}_{p_n,p^{\partial}_n;f}(M,\partial M, [g])\geq Y^{a_2,b}_{p_n,p^{\partial}_n;f}(M,\partial M, [g]).
\end{equation}
Indeed, for a positive $u\in C^{\infty}(M)$, let $\bar{u}=\lambda_{a_1,b}u$. Then, 
\[a_2\int_M\bar{u}^{p_n}dv_g+b\int_{\partial M}\bar{u}^{p^{\partial}_n}d\sigma_g=(a_2-a_1)\int_M\bar{u}^{p_n}dv_g+1\geq 1.\]
This implies that $\lambda_{a_2b}\leq 1$ for $\bar{u}$. Then for any  $u\in C^{\infty}(M)$ we have that  
\[E(\lambda_{a_2b}\bar{u})\leq E(\bar{u})\]
which implies Inequality \eqref{monoticity of Yab}. 

Similarly, if $b_2\geq b_1$, we have that \[Y^{a,b_1}_{p_n,p^{\partial}_n;f}(M,\partial M, [g])\geq Y^{a,b_2}_{p_n,p^{\partial}_n;f}(M,\partial M, [g]).\]

Let $a=1$ and $b=0$, then $Y^{1,0}_{p_n,p^{\partial}_n;f}(M,\partial M, [g])=Y_f(M,\partial M,g)$, and $c_{1,0}=0$.

By (Proposition 3.2, \cite{Escobar5}) we known that the function $b\longrightarrow Y^{a,b}_{p_n,p^{\partial}_n;f}(M,\partial M, [g])$  is continuous.

Let $b=1$, by  Inequality \eqref{cotainf cte } and by \eqref{monoticity of Yab}, we have that $\lim_{a\to 0^+} c_{ab}=+\infty$.

Let $a=1$, then \[\lim_{b\to 0^+} Y^{1,b}_{p_n,p^{\partial}_n;f}(M,\partial M, [g])=Y_f(M,\partial M,g).\]  
Therefore, by Inequality \eqref{cotasup cte} we have that $\lim_{b\to 0^+} c_{1b}=0$.

Let $b=-1$, it follows from \eqref{monoticity of Yab} that $\lim_{a\to 0+}c_{a,-1}=-\infty$
and $\lim_{a\to +\infty}c_{a,-1}=0^-$.

\end{proof} 

\section{Proof of  Theorem \ref{localbifurcation}}\label{Section4}

Let $(M,g)$ be a Riemannian manifold with boundary with constant scalar curvature and minimal boundary and let $f$ be an isoparametric function. Let us assume for simplicity that $\partial M$ is connected. Consider 
$$X:=A\times \R_{>0}$$ where $A$ is the space of positive functions with zero normal derivative at the boundary that also belong to $S_f$.  For $s>2$, let  $F^s:X\longrightarrow C^{2,\alpha}(M)\cap S_f$  be the map defined by 
\begin{equation}\label{bifurcationmap1}
F^s(u,\lambda):=\Delta_gu+\lambda\big(u-u^{s-1}).
\end{equation}
Note that $F^s(1,\lambda)=0$ for any $\lambda>0$. The points of the form $(1,\lambda)$ with $\lambda\in \R_{>0}$ are called  trivial zeroes of $F$ and $\{1\}\times \R_{>0}$ the axis of trivial zeroes.

The linearization of $F^s$ at a given  trivial zero $(1,\lambda_0)$ is 
$$D_uF^s(1,\lambda_0)(v)=\Delta_gv+\lambda_0(2-s)v.$$
A function  $v$ belongs to the kernel of $D_uF^s(1,\lambda_0)$ if and only if  $v\in S_f$ and satisfies the following Neumann boundary value problem
\begin{equation}\label{Linarizedproblem}\begin{cases}  \Delta_gv=\lambda_0(s-2)v & \mbox{on } M, \\ \frac{\partial v}{\partial \eta}=0, & \mbox{on } \partial M.  
\end{cases}
\end{equation}
That is,  $v\in S_f$ should be an eigenfunction with associate eigenvalue $\lambda_0(s-2)$.

It is well known  that the spectrum of the Neumann eigenvalue problem restricted to $S_f$ is a non-bounded  sequence   \begin{equation*}\label{spectrum}0=\mu_0^{f,\mathcal{N}}<\mu_{1}^{f,\mathcal{N}}\leq \mu_{2}^{f,\mathcal{N}}\leq \dots\leq \mu_{k}^{f,\mathcal{N}}\nearrow +\infty.
\end{equation*}

Let $E_{i}^{f,\mathcal{N}}$ denotes the space of eigenfunctions with associated eigenvalue $\mu_i^{f,\mathcal{N}}$.

Let $v=\varphi\circ f\in S_f$. By a straightforward computation we see that $v$ is a Neumann eigenfunction  with  associated eigenvalue  $\mu$  if and only if  $\varphi$ fulfills
\begin{equation*}\begin{cases}  -b(f)\varphi''(f)+a(f)\varphi'(f)=  \mu \varphi (f)& \mbox{on } M, \\ \varphi'(f)\frac{\partial f}{\partial \eta}=0 & \mbox{on } \partial M, \end{cases}
\end{equation*}
where $a$ and $b$ are the functions that appear in the Equations \eqref{Cgradient} and   \eqref{Claplacian}, respectively.
Therefore, $v\in E_{i}^{f,\mathcal{N}}$ if and only if  $\varphi$ satisfies 
\begin{equation}\label{eigenvalue}\begin{cases}  -b(t)\varphi''(t)+a(t)\varphi'(t)- \mu_i^{f,\mathcal{N} }\varphi (t)=0 &  t\in [t_-,t_+]-\{r\} , \\ \varphi'(r)=0 &  \end{cases}
\end{equation}
where $r$ is the value that $f$ takes on  $\partial M$. This implies that $\dim(E_{i}^{f,\mathcal{N}})=1$ for any $i$. We have the following proposition.

\begin{Proposition}\label{condKer=1} The dimension of the kernel of $D_uF^s(1,\frac{\mu_i^{f,\mathcal{N}} }{s-2})$  is 1.
	\end{Proposition}

Let $\mu_i^{f,\mathcal{N}}$ be an eigenvalue of the Neumann problem and let $v_i$ such that $E_{i}^{f,\mathcal{N}}=span(v_i)$. Since $D_uF^s(1,\frac{\mu_i^{f,\mathcal{N}} }{s-2})$ is a self-adjoint operator we see that  
\begin{equation}\label{Cond0}
\int_M  D_uF^s(1,\frac{\mu_i^{f,\mathcal{N}}}{s-2})(v)v_idv_g=0
\end{equation}
for any $v\in S_f$ with $\frac{\partial v}{\partial \eta}=0$. Hence,  
\begin{equation}\label{Cond2}
V:=Range\Big(D_{u}F^s(1,\frac{\mu_i^{f,\mathcal{N}}}{s-2})\Big)\ \mbox {has codimension}\ 1.
\end{equation}

One can check  that 
\begin{equation}\label{Cond3}
D_{\lambda}F^s(1,\frac{\mu_i^{f,\mathcal{N}}}{s-2})=0
\end{equation}
 and
 \begin{equation}\label{Cond4} 
 D_{\lambda \lambda}F^s(1,\frac{\mu_i^{f,\mathcal{N}}}{s-2})=0.
 \end{equation}  
Also, using that 
$\int_M D_{u\lambda}(1,\frac{\mu_i^{f,\mathcal{N}}}{s-2})(v_i) v_idv_g \neq 0$ and Equation \eqref{Cond0} we obtain that 
\begin{equation}\label{Cond5}D_{u\lambda}F^s(1,\frac{\mu_i^{f,\mathcal{N}}}{s-2})(v_i)\notin V.
\end{equation}

Since Proposition \ref{condKer=1},  \eqref{Cond2}, \eqref{Cond3}, \eqref{Cond4} and \eqref{Cond5} we can use the  Crandall-Rabinowitz's local bifurcation theorem   
 (see for instance \cite{Crandall} or \cite{Nirenberg}) to conclude that a trivial zero $(1,\lambda)$ is a bifurcation point of $F^s$ if and only if  
\begin{equation}\label{CONDBIFURCATIONPOINT}
\lambda=\frac{\mu_i^{f,\mathcal{N}}}{s-2}.
\end{equation}

This means that there is a curve $\alpha$ in $X$ that cuts transversally  the  axis of trivial zeroes at $(1,\frac{\mu_i^{f,\mathcal{N}}}{s-2})$. More precisely, for $r$ small enough $\alpha$ is parametrized by 
\begin{equation}\label{curvabifurcacion}\alpha(r)=\Big(rv_i+O(r^2), \beta(r)\Big)
\end{equation}
with  $\beta(0)=\frac{\mu_i^{f,\mathcal{N}}}{s-2}$.

Let us recall the assumptions of Theorem  \ref{localbifurcation}. For $t>0$,  $(M\times N,g(t)=g+th)$ is the Riemannian product where  $(M^m,g)$  is a compact Riemannian manifold ($m\geq 2$) with boundary, constant scalar curvature and minimal boundary and $(N,h)$ is any closed Riemannian manifold with positive constant scalar curvature.

\begin{proof}[Proof of Theorem \ref{localbifurcation}] 
	 We normalize Equation \eqref{EYM} by taking $c=s_{g(t)}$. Then a positive solution $u$ of Equation \eqref{EYM} that depends only on $M$ satisfies 
	 \begin{equation}\label{THMBIFEQ}\begin{cases} \Delta_{g}u+\lambda(t)u=\lambda(t) u^{p_{m+n}-1} & \mbox{on } M, \\ \frac{\partial u}{\partial \eta}=0 & \mbox{on } \partial M, \end{cases}
	 \end{equation}
	 where 
	 \begin{equation*}
	 \label{parametro} \lambda(t)=\frac{s_{g}+t^{-1}s_{h}}{a_{m+n}}
	 \end{equation*}
	 Recall that $u^{p_{m+n}-2}g(t)$ is a metric of constant scalar curvature and minimal boundary. Note that since the scalar curvature of $g(t)$ is constant and $h_{g(t)}=0$ on $\partial(M\times N)$ 
	 $u\equiv 1$ is a trivial solution of Equation \eqref{THMBIFEQ}  for any $t>0$.
	 
	  Let consider the map $F^{p_{m+n}}$. By \eqref{CONDBIFURCATIONPOINT} we get  that $(1,\lambda)$ is a bifurcation point of $F^{p_{m+n}}$ if and only if  
	 \[\lambda=\frac{(m+n-2)\mu_i^{f,\mathcal{N}}}{4}.\]
	Therefore,  if  \[t_i=\frac{s_h}{\mu_i^{f,\mathcal{N}}(m+n-1)-s_g},\]
	whenever $\mu_i^{f,\mathcal{N}}(m+n-1)\neq s_g$,  $(1,\lambda(t_i))$ is a bifurcation point. Note that since  $\{\mu_i^{f,\mathcal{N}}\}_{i\in \N}$ is not bounded from above the sequence $\{t_i\}$  tends to 0 as $i$ goes to $+\infty$.
By the existence of the curve $\alpha$ (see \eqref{curvabifurcacion})   we know that given $\varepsilon_i>0$ there exists $\gamma_i$ such that $|t_i-\gamma_i|<\varepsilon_i$  such that there exists a positive smooth solution $u_{\gamma_i}\in S_f$ of Equation \eqref{THMBIFEQ} with $\lambda(\gamma_i)$.

 \end{proof}


\begin{thebibliography}{aa}
\bibitem{Adams} R. Adams, J. Fournier, {\it Sobolev spaces}, 2nd ed., Academic Press, New Tork, NY, (2003), pp.xiii+320.

	
\bibitem{Almaraz} S. Almaraz, {\it An existence theorem of conformal scalar-flat metrics on manifolds with boundary}, Pacific journal of mathematics, {\bf 248}  (2010), 1:1-22. 
\bibitem{Almaraz2} S. Almaraz, {\it Convergence of scalar-flat metrics on manifolds with boundary under a Yamabe-type flow}, J. Differ. Equations {\bf 259} (2015), 7: 2626-2694. 

\bibitem{Araujo} H. Ara\'ujo, {\it Existence and compactness of minimizers of the Yamabe problem on manifolds with boundary}, Comm. Anal. Geom. {\bf 12} (2004) 3:487-510.

	\bibitem{Brendle_Chen} S. Brendle, Chen {\it An existence theorem for the Yamabe problem on manifolds with boundary},  J. Eur. Math. Soc. {\bf 16} (2014), 991–1016.
	
	\bibitem{Cartan} E. Cartan, {\it Familles de surfaces isoperimetriques dans les
	espaces  a courbure constante}, Ann. Mat. Pura Appl. {\bf 17} (1938), 177-191.
	
	\bibitem{Cartan2} E. Cartan, {\it Sur des familles remarquables d\'hypersurfaces isoparam\'etriques dans les espaces sph\'eriques}, Math. Z. {\bf 45} (1939), 335-367.
	
	\bibitem{Cecil} T. Cecil, P. Ryan , {\it Geometry of Hypersurfaces}, 
	Springer Monographs in Mathematics, New York,  xi+596 p. (2015).
	

	
	\bibitem{Chi1} Q.-S. Chi,{\it Isoparametric hypersurfaces with four principal curvatures IV}, J. Diff. Geom., {\bf 115} (2020), 2:225-301.
	

\bibitem{Chi2} Q.-S. Chi,{\it The isoparametric story, aheritage of \'Elie Cartan}, arXiv:2007.02137

 \bibitem{Chen} S. Chen, {\it Conformal deformation to scalar flat metrics with constant mean curvature on the boundary in higher dimensions}, https://arxiv.org/abs/0912.1302. 

\bibitem{Chen-Ho-Sun}  X. Chen, P. Ho  L. Sun,  {\it
	Prescribed scalar curvature plus mean curvature flows in compact manifolds with boundary of negative conformal invariant}, Ann. Global Anal. Geom. , {\bf 53} (2018), 1:121-150.
	
	\bibitem{Chen-Sun}  X. Chen, L. Sun,  {\it Existence of conformal metrics with constant scalar curvature and constant bound- ary mean curvature on compact manifolds}, Commun. Contemp. Math. 21  {\bf3 } (2019) 1850021,  51 pp.
	
	\bibitem{Chen-Ruan- Sun} X. Chen, Y. Ruan, L. Sun, {\it The Han-Li conjecture in constant scalar curvature and constant boundary mean curvature problem on compact manifolds},      Adv. Math.  {\bf358}   (2019) 106854.
	
	\bibitem{Cherrier}  P. Cherrier,  {\it Problemes de Neumann non lin\'eaires
		sur les  vari\'et\'es riemanniennes} ,  J. Funct. Anal. {\bf 57} (1984) 154-206.
	
	
	
	
	\bibitem{Crandall} M. G. Crandall, P. H. Rabinowitz {\it Bifurcation from simple eigenvalues} J. Funct. Anal.  (1971),  8:321-340.	
	
	
	
	
	  
	
	
	
	

	\bibitem{delaParra}  A. B. de la Parra, J. J. Batalla, J.  Petean, {\it 
	Global bifurcation techniques for Yamabe type equations on Riemannian manifolds}, 	Nonlinear Anal., Theory Methods Appl.,  Ser. A,  Theory Methods {\bf 202},  (2021) 23 p.
	
	
	\bibitem{Escobar1} J. F. Escobar, {\it  Conformal deformation of a Riemannian metric to a scalar flat metric with constant mean
		curvature on the boundary},  Ann.  Math., {\bf 136} (1992), 1–50. 
	
	\bibitem{Escobar3} J. F. Escobar, {\it Addendum: Conformal deformation of a Riemannian metric to a scalar flat metric with constant mean curvature. }, Ann. Math., {\bf 139} (1994),  3: 749-750.
	
	\bibitem{Escobar2} J. F. Escobar, {\it The Yamabe problem on manifolds with boundary}, J. Diff. Geom. {\bf 35} (1992), 21–84.
	
	
	\bibitem{Escobar4} J. F. Escobar, {\it Conformal metrics with prescribed mean curvature on the boundary},  Calc. Var. Partial Differential Equations {\bf 4}  (1996), 6:559–592.

 \bibitem{Escobar5} J. F. Escobar,{\it Conformal deformation of
a Riemannian metric to a constant scalar curvature metric with
constant mean curvature on the boundary}, Indiana University Mathematics Journal,  {\bf 45}, (1996) 4:917-943.
	
	\bibitem{Fernandez-Petean} J. C. Fern\'andez, J. Petean, {\it Low energy nodal solutions of the Yamabe equation on the sphere}. J. Differ. Equations {\bf 268} (2020) 11-6576-6597.       
	

	
	\bibitem{Ge-Tang} J. Ge, Z.  Tang, {\it Isoparametric functions and exotic spheres}, J. Reine Angew. Math., {\bf 683} (2013), 161-180.
	
	
	\bibitem{Ge-Tang2} J. Ge,  Z.  Tang, {\it Geometry of isoparametric hypersurfaces in Riemannian	manifolds}, Asian J. Math., {\bf 18} (2014), 1:117-125.
	
	\bibitem{Han-YYli} Z. C. Han, Y. Y. Li, {\it The Yamabe problem on manifolds with boundary: existence and compactness results},  Duke Math. J. {\bf 99} (1999), 3:489-542.

 \bibitem{Han-YYli2} Z. C. Han, Y. Y. Li, {\it The existence of conformal metrics with constant scalar curvature and
constant boundary mean curvature}, Comm. Anal. Geom. {\bf 8} (2000) 4:809-869.
	
	\bibitem{Hebey} E. Hebey, {\it Nonlinear analysis on manifolds: {S}obolev spaces and
		inequalities}, Courant Lecture Notes in Mathematics, {\bf 5}, New York University, Courant Institute of Mathematical Sciences, New York; AMS, Providence, RI, (1999), pp.x+309.
	
     
	
	
	\bibitem{Henry} G. Henry, {\it Isoparametric functions and nodal solutions of the Yamabe equation},  Ann. Glob. Anal. Geom. {\bf 56} (2019), 2:203–219. 
	
	
	
	\bibitem{Henry-Petean} G. Henry, J. Petean, {\it Isoparametric hypersurfaces and metrics of constant scalar
		curvature}, Asian J. Math., {\bf 18} (2014), 1:53-67.
	           
          
	
	             
	
	


	
	

	
	

	
	
	\bibitem{Mayer-Ndiaye} M. Mayer, C. B. Ndiaye, {\it Barycenter technique and the Riemann mapping problem of Cherrier–Escobar}, 
	J. Differential Geom.  {\bf 107}  (2017),  3:519-560.
	
	
	
	
	\bibitem{Miyaoka} R. Miyaoka, {\it Transnormal functions on a Riemannian manifold}, Differential Geom. Appl.,  {\bf 31}     (2013),  1:130-139.
	
	
	    
	
	
 
	
	
	\bibitem{Marques} F. C.  Marques, {\it Existence results for the Yamabe problem on manifolds with boundary}, Indiana Univ. Math. J. {\bf 54} 
	(2005) 6:1599–1620.
	
	
		\bibitem{Marques2} F. C. Marques, {\it Conformal deformations to scalar flat metrics with constant mean curvature on the boundary}, Comm. Analysis Geom. {\bf 15} (2007) 2:381-405
	
	\bibitem{Nirenberg} L. Nirenberg, Topics in nonlinear functional analysis, New York University Lecture Notes, New York, 1974.
	
	\bibitem{Qian-Tang} C. Qian, Z. Tang, {\it Isoparametric functions on exotic spheres}, Adv. Math., {\bf 272}, (2015), 611-629. 
	
	\bibitem{Qianetal} C. Qian, Z.  Tang, W. Yan {\it New examples of Willmore submanifolds in the unit sphere via isoparametric functions. II.} 	Ann. Global Anal. Geom.,  {\bf 43}  (2013), 1:47-62.


	
	\bibitem{Segre} B. Segre, {\it Famiglie di ipersuperficie superficie isoparametrische negli spazi euclidei ad un 
	qualunque numero di demesioni}, Atti. Accad. naz. Lincei. Rend. Cl. Sci. Fis. Mat. Natur {\bf 27} (1938), 203-207.
	
	
	
	\bibitem{Tang-Yan1} Z. Tang, W. Yan {\it
 
	New examples of Willmore submanifolds in the unit sphere via isoparametric functions},
	Ann. Global Anal. Geom., {\bf 42} (2012) 3:403-410.
	
	
	\bibitem{Tang-Yan2} Z. Tang, W. Yan {\it Isoparametric foliation and Yau conjecture on the first eigenvalue}, 
	J. Differ. Geom.,  {\bf 94} (2013),  3:521-540.
	
	
	
	\bibitem{Wang} Q. M. Wang, { \it Isoparametric functions on Riemannian Manifolds. I}, Math. Ann. {\bf{277}}, (1987), 639-646.
	
	
	\bibitem{Xu} J. Xu , {\it The Boundary Yamabe Problem, I: Minimal Boundary Case}, https://arxiv.org/abs/2111.03219
	 
	\bibitem{Yau} S. T. Yau, {\it Open problems in geometry}, Lectures on Differential Geometry, R. Schoen and S. T. Yau, Conference Proceedings and Lectures Notes in Geometry and Topology, {\bf 1},  International Press, (1993), Boston.
	  
	 
\end{thebibliography}
\end{document}